\documentclass[12pt,reqno]{amsart}
\usepackage{amsmath}
\usepackage{amssymb}
\usepackage[left=3cm,top=3cm,right=3cm,bottom=3cm]{geometry}
\usepackage{epsfig}
\usepackage[normalem]{ulem}
\usepackage[usenames,dvipsnames]{color}
\usepackage{todonotes}
\usepackage[colorlinks, citecolor=blue, linkcolor=blue, urlcolor=blue]{hyperref}
\usepackage{soul}

\usepackage{tikz}
\usetikzlibrary{backgrounds}
\usetikzlibrary{intersections}
\usepackage{floatrow}

\newtheorem{theorem}{Theorem}[section]
\newtheorem{lemma}[theorem]{Lemma}

\newcommand\eps{\varepsilon}
\newcommand{\E}{\mathbb E}

\newcommand{\Prob}{\mathbb{P}}
\newcommand{\Bin}{\mathrm{Bin}}

\newcommand{\Nn}{{\mathbb N}}
\newcommand{\Zz}{{\mathbb Z}}
\newcommand{\Rr}{{\mathbb R}}

\title{Broadcasting on Paths and Cycles}
\author[R.~Huq \and P.~Pra\l{}at]
{Reaz Huq \and Pawe\l{} Pra\l{}at}

\address{Ryerson University, Toronto, Canada}
\email{reaz.huq@ryerson.ca, pralat@ryerson.ca}

\date{\today}

\begin{document}

\maketitle

\begin{abstract}
Consider the following broadcasting process run on a connected graph $G=(V,E)$. Suppose that $k \ge 2$ agents start on vertices selected from $V$ uniformly and independently at random. One of the agents has a message that she wants to communicate to the other agents. All agents perform independent random walks on $G$, with the message being passed when an agent that knows the message meets an agent that does not know the message. The broadcasting time $\xi(G,k)$ is the time it takes to spread the message to all agents. We provide tight bounds for $\xi(P_n,k)$ and $\xi(C_n,k)$ that hold asymptotically almost surely for the whole range of the parameter~$k$. 
\end{abstract}

\section{Introduction}

In this paper, we investigate the problem of broadcasting messages between agents that randomly move on a connected graph $G=(V,E)$. The assumption is that $k \ge 2$ agents start the process at random locations on the graph and then perform a random walk along its vertices. One agent, selected in advance, initially possesses some information. If at some point during the process two agents meet at some vertex or pass each other at some edge and only one of them possesses the information, it is passed along to the other agent. The broadcasting time $\xi(G,k)$ is the time it takes to spread the message to all agents. (Formal definition will be provided in Section~\ref{sec:problem}.)

\medskip

The performance of a random walk in a network is a fundamental process that has found applications in many areas of computer science. Since this paper contains theoretical results, we will focus on prior results of related processes that were investigated rigorously and in a theoretical context. As this is still a very broad topic, we only scratch the surface and focus on multiple random walks performed simultaneously (which has many applications in distributed computing, such as sampling). For more on other directions, we direct the reader to one of the many books on Markov chains; see, for example~\cite{book_Yuval}.

Suppose there are $k \ge 2$ particles, each making a simple random walk on a graph $G$. Even if the particles are oblivious of each other, it is important and non-trivial to estimate the (vertex) cover time, an extensively studied graph parameter that is defined as the expected time required for the process to visit every vertex of $G$. The first paper~\cite{ref1} on this problem was concerned with the walk starting on the worst case vertices and subsequent papers~\cite{ref2,ref11} dealt specifically with starting positions selected randomly from the stationary distribution. Questions become more interesting (and difficult) once we allow particles to interact once they meet. We assume that interaction occurs only when meeting at a vertex, and that the random walks made by the particles are otherwise independent. There are at least four interesting variants of this process:
\begin{itemize}
    \item \emph{Predator-Prey}: estimate the expected time-to-extinction of the prey particles under the assumption that $k$ predator and $\ell$ prey particles walk independently; predators eat prey particles upon meeting at a vertex.
    \item \emph{Coalescing particles}: estimate the expected time to coalesce to a single particle under the assumption that $k$ particles walk independently and coalesce upon meeting at a vertex.
    \item \emph{Annihilating particles}: estimate the expected time-to-extinction of all particles under the assumption that $k = 2\ell$ particles walk independently and destroy each other (pairwise) upon meeting at a vertex.
    \item \emph{Talkative particles}: estimate the expected time to broadcast a message---this is exactly the problem that we are concerned with in this paper.
\end{itemize}

All of these variants have been studied for random $d$-regular graphs $\mathcal{G}_{n,d}$~\cite{Alan_d-reg}. In particular, if $d \ge 3$ is a fixed constant and $k \le n^{\eps}$ for a sufficiently small constant $\eps>0$, then asymptotically almost surely (see the next section for a definition and a notation used)
$$
\xi(\mathcal{G}_{n,d}, k) \sim \frac {2 H_{k-1}}{k} \cdot \frac {d-1}{d-2} \cdot n.
$$
Moreover, in a recent paper the authors of this paper (along with three co-authors) provided a complete characterization of $\xi(K_n,k)$ for the whole range of the parameter~$k$. Interestingly, $\xi(K_n,k)$ is well concentrated around $2n \ln k / k$ for a wide range of possible values of $k$, but the behaviour changes when $k$ is very large, namely, when $k$ is linear in $n$~\cite{cliques}.  These are the only theoretical results on the broadcasting time that we are aware of. However, the \emph{frog model}, a well-known and well-studied epidemic model, is somewhat related to our problem~\cite{Itai}. There are a few differences between the two models. For example, in the frog model the number of agents that start on a given vertex is an independent Poisson random variable, some agents do not perform walks, and agents have a given lifespan. Though variations of the frog model have been studied (for example, in~\cite{Popov}, where every agent performs a random walk), we have seen none which are direct analogues of the process we study in this paper.

\medskip

On the other hand, the variant of coalescing particles is very well-studied, mainly because of its surprising connection to the voter model~\cite{voter2,voter6}. The state of the process at a given time $t$ is described by a function $\mu_t : V(G) \to O$, where $V(G)$ is the vertex set of a graph $G$ and $O$ is a given set of possible opinions. Each vertex $v\in V(G)$ ``wakes up'' at rate 1. When it wakes up at a time $t>0$, $v$ chooses one of its neighbours $w$ uniformly at random and updates its value $\mu_t(v)$ to the opinion of $w$; all other opinions remain the same. A classical duality result (see, for example,~\cite{voter2,voter6}) directly relates the state of the process at a given time to a system of coalescing random walks on $G$ moving backwards in time. As we already mentioned, there are many interesting results on the coalescing time. Let us only mention a beautiful conjecture posed by Aldous and Fill in the mid-nineties (Open problem 13, Chapter~14 of~\cite{voter2}). They conjectured an upper bound for the mean coalescent time in terms of the mean hitting time of a single random walk. The conjecture was proved in~\cite{Roberto12,Roberto13}.

\medskip

Let us now briefly discuss the following well-known and well-studied rumour spreading protocols: \emph{Push} and \emph{Push~\&~Pull}. Suppose that one vertex in a network is aware of a piece of information, the ``rumour'', and wants to spread it to all vertices. In each round of the \emph{Push} protocol, every informed vertex contacts a random neighbour and sends the rumour to it (``pushes'' the rumour). In \emph{Push~\&~Pull}, uninformed vertices can also contact a random neighbour to get the rumour if the neighbour knows it (``pulls'' the rumour).

There is a long sequence of interesting and important papers studying the runtime of \emph{Push} on the complete graph. The first paper considering this protocol is~\cite{Push10} but more precise bounds were provided in~\cite{Push17} and then in~\cite{Push5}, in which it was shown that the process can essentially be stochastically bounded (from both sides) by coupon collector-type problems. A very recent paper~\cite{Latin} both determines the limiting distribution and explains why it is a difficult problem: the runtime, scaled appropriately by $(\log_2 n + \ln n)$, has no limiting distribution; instead, it exhibits double-oscillatory behaviour. \emph{Push} has been extensively studied on several other graph classes besides complete graphs. 

The \emph{Push~\&~Pull} protocol has an equally long sequence of interesting papers studying it. The synchronous version of the protocol (as described above) was introduced in~\cite{Pull5} and popularized in~\cite{Pull23}. However, such synchronized models (that is, models in which all vertices take action simultaneously at discrete time steps) are not plausible for many applications, including real-world social networks. As a result, an asynchronous version of the model with a continuous timeline was introduced in~\cite{Pull4}. In this variant, each vertex has its own independent clock that rings at the times of a rate 1 Poisson process with the protocol specifying what a vertex has to do when its own clock rings. The first theoretical relationships between the spread times in the two variants was provided in~\cite{Wormald}.

\medskip

In this paper, we focus on paths and cycles. We show that the behaviour of $\xi(G, k)$ is similar for both families of graphs and is approximately equal to $n^2 / k$ (that is, up to poly-log factors), provided that $k \le n$. (See Theorem~\ref{thm:main} for the precise statement.) The paper is structured as follows. In the next section, we formally define our problem, introduce asymptotic notation, and state the main result. The whole of Section~\ref{sec:proofs} is devoted to proving the main result.

\section{Formulation of the Problem and the Main Result}\label{sec:problem}

In this section, we formally define the process we aim to analyze (Subsection~\ref{sec:proglem}). Though we define it for any connected graph, in this paper we focus on paths and cycles. As our results are asymptotic in nature, we need to introduce the asymptotic notation that is used throughout the entire paper (Subsection~\ref{sec:asymptotic}). Finally, we state the main result that combines all ranges for the number of agents involved (Subsection~\ref{sec:main_result}).

\subsection{Problem}\label{sec:proglem}

Suppose that we are given a connected graph $G=(V,E)$ on $n=|V|$ vertices, and let $k \ge 2$ be any natural number. There are $k$ agents, one of which is green with the rest being white. The process starts at round $t=0$ with agents located randomly on vertices of $G$; that is, each agent starts at any vertex $v \in V$ with probability $1/n$, independently of other agents and independently of her colour. Each agent synchronously performs an independent random walk, regardless of whether she is green or white. In other words, an agent occupying vertex $v \in V$ moves to any neighbour of $v$ with probability equal to $1/\deg(v)$. A white agent becomes green when she meets a green agent at some \emph{round} $t \ge 0$. In particular, all agents that start at the same vertex as the initial green agent become green at the very beginning. For most graphs it does not make a substantial difference but, for example, for bipartite graphs this definition has a flaw: if $G$ is a bipartite graph with parts $X, Y$, a given white agent and the green agent will not meet if they do not start on the same partition of the graph. In order to solve this potential issue, we will allow a white agent to become green if they move from $u$ to $v$ at the same round that a green agent moves from $v$ to $u$. Let $\xi = \xi(G, k)$ be the time it takes for all agents to become green. (Note that $\xi$ is a random variable even when $G$ is a deterministic graph.)

\medskip

We say that the process is at \emph{phase} $\ell$ ($1 \le \ell \le k$) if there are $\ell$ green agents (and so $k-\ell$ white agents). The first phase is usually phase $1$, unless some white agents start at the same vertex as the green agent but this is rare if $k$ is small. Clearly, the process always moves from a smaller phase to a larger phase. If $k$ is small, then typically it takes some number of rounds for the process to move to another phase though some phases may be skipped. If $k$ is large, then skipping phases is quite common. The process ends at the end of round $\xi$, when we are about to move to phase $k$.

\subsection{Asymptotic Notation}\label{sec:asymptotic}

Our results are asymptotic in nature, that is, we will assume that $n\to\infty$. Formally, we consider a sequence of graphs $G_n=(V_n,E_n)$ (paths and cycles on $n$ vertices) and $k=k(n)$ may be a function of $n$ that tends to infinity as $n\to \infty$. We are interested in events that hold \emph{asymptotically almost surely} (\emph{a.a.s.}), that is, events that hold with probability tending to 1 as $n\to \infty$.

Given two functions $f=f(n)$ and $g=g(n)$, we will write $f(n)=O(g(n))$ if there exists an absolute constant $c \in \Rr_+$ such that $|f(n)| \leq c|g(n)|$ for all $n$, $f(n)=\Omega(g(n))$ if $g(n)=O(f(n))$, $f(n)=\Theta(g(n))$ if $f(n)=O(g(n))$ and $f(n)=\Omega(g(n))$, and we write $f(n)=o(g(n))$ or $f(n) \ll g(n)$ if $\lim_{n\to\infty} f(n)/g(n)=0$. In addition, we write $f(n) \gg g(n)$ if $g(n)=o(f(n))$ and we write $f(n) \sim g(n)$ if $f(n)=(1+o(1))g(n)$, that is, $\lim_{n\to\infty} f(n)/g(n)=1$.

Finally, for any $\ell \in \Nn$ we will use $[\ell]$ to denote the set of $\ell$ smallest natural numbers, that is, $[\ell] := \{1, 2, \ldots, \ell\}$.

\subsection{Main Result}\label{sec:main_result}

Let us summarize the main results for paths and cycles in one theorem. More detailed and stronger statements can be found in the next section.

\begin{theorem}\label{thm:main}
Let $\omega=\omega(n)$ be any function that tends to infinity as $n \to \infty$. 
Let $G=P_n$ (a path on $n$ vertices) or $G=C_n$ (a cycle on $n$ vertices).
Depending on the parameter $k=k(n)$, the following properties hold a.a.s.:
\begin{itemize}
\item [(a)] If $k \le \omega$, then 
$$
\frac {n^2}{\omega^3 \ln \omega} \le \xi(G,k) \le n^2 \omega.
$$  
\item [(b)] If $\omega < k \le \omega \ln n$, then 
$$
\frac {n^2}{\omega k^2 \ln k} \le \xi(G,k) \le n^2 \omega.
$$
\item [(c)] If $\omega \ln n < k \le n / \ln^2 n$, then 
$$
\Omega \left( \frac {n^2}{k (\ln k) (\ln n)} \right) = \xi(G,k) = O \left(  \frac {n^2 \ln n}{k} \right).
$$
\item [(d)] If $n / \ln^2 n < k \le 50 n \ln n$, then 
$$
\Omega \left( n \right) = \xi(G,k) = O \left(  \frac {n^2 \ln n}{k} \right).
$$
\item [(e)] If $k = k(n) \ge 50 n \ln n$, then $\xi(G,k) = \Theta(n)$. 
\end{itemize}
In any case, if $k = n^{x+o(1)}$ for some $x \in [0,1]$, then a.a.s.\ 
$$
\xi(G,k) = \frac {n^{2+o(1)}}{k} = n^{2-x+o(1)}.
$$
\end{theorem}

Parts~(a) and~(b) follow from Theorems~\ref{thm:upper_bound_small_k} and~\ref{thm:small_k_upperbound}.
Parts~(c) and~(d) follow from Theorems~\ref{thm:medium_k_lowerbound} and~\ref{thm:medium_k_upperbound}.
Finally, part~(e) follows from Theorem~\ref{thm:large_k}.
The arguments for paths $P_n$ and cycles $C_n$ are almost the same. Because of the boundary effect, the argument for paths is usually slightly more challenging. 
In order to avoid reproving theorems for the two classes, we provide a coupling that shows that $\xi(P_n,k) \le \xi(C_{2(n-1)},k)$, provided that $k=o(n)$---see Lemma~\ref{lem:coupling}. It is a simple but interesting and useful observation as it proves that to establish the asymptotic behaviour for both classes, one only needs to prove upper bounds for cycles and lower bounds for paths. Having said that, we have to admit that we more often than not make exceptions to this rule. For example, the argument for large values of $k$ (Section~\ref{pathslargek}) is the same for both classes of graphs so there is no need for coupling which does not apply to this range of values of $k$ anyway. The proofs for the corresponding upper bounds for small values of $k$ (Section~\ref{sec:small_values_of_k}) are substantially different for paths and cycles, so we decided to include both arguments. (Of course, the result for paths is alternatively implied by the coupling.) Finally, since part of the range of medium values of the parameter $k$ (Section~\ref{sec:medium_values_of_k}) is not covered by the coupling, we decided to present an argument for paths that is simpler and only mention straightforward adjustments to cycles.

\section{Proofs}\label{sec:proofs}

This whole section is devoted to proving Theorem~\ref{thm:main}. We investigate the process running on $P_n$, the path on $n$ vertices, and $C_n$, the cycle on $n$ vertices. It will be convenient to label vertices of $P_n$ as follows: $V(P_n) = [n] := \{1, 2, \ldots, n\}$ and $E(P_n) = \{ i (i+1) : i \in [n-1]\}$. Similarly, $V(C_n) = [n] := \{1, 2, \ldots, n\}$ and $E(C_n) = \{ i (i+1) : i \in [n-1]\} \cup \{1n\}$.

The proofs require unique approaches depending on the number of agents involved (parameter $k = k(n)$). Thus, we will deal with each sub-range of $k$ independently. However, before we start, let us state some concentration inequalities that we will use often, and introduce the coupling between the processes run on paths and cycles.

\subsection{Chernoff inequality}

Throughout the paper, we will be using the following concentration inequality: let $X \in \textrm{Bin}(n,p)$ be a random variable characterized as a binomial distribution with parameters $n$ and $p$. Then, a consequence of \emph{Chernoff's bound} (see e.g.~\cite[Corollary~2.3]{JLR}) is that
\begin{equation}\label{eq:chern}
\Prob( |X- \E [X] | \ge \eps \, \E [X]) \le 2\exp \left( - \frac {\eps^2 \E [X]}{3} \right)
\end{equation}
for  $0 < \eps < 3/2$.

\subsection{Hoeffding-Azuma inequality}

Let $X_0, X_1, \ldots$ be an infinite sequence of random variables that is a martingale; that is, for any $a \in \Nn$ we have $\E [X_{a} | X_{a-1}] = X_{a-1}$. Suppose that there exist constants $c_a > 0$ such that $|X_a-X_{a-1}| \le c_a$ for each $a \le t$. Then, the \emph{Hoeffding-Azuma inequality} implies that for every $b > 0$,
\begin{equation}\label{eq:h-a}
\Prob( \exists i (0 \le i \le t) : |X_i - X_0| \ge b) \le 2\exp \left( - \frac {b^2}{2 \sum_{a=1}^t c_a^2} \right).
\end{equation}

\subsection{Coupling}

We will show now that one may couple the processes run on paths and cycles. This coupling will allow us to translate bounds obtained for one class to another one. The argument applies provided that $k=o(n)$, and the coupling can only be established a.a.s.\ but it is enough as our main result holds a.a.s.\ anyway.

\begin{lemma}\label{lem:coupling}
Suppose that $k = o(n)$. The processes on $P_n$ and $C_{2(n-1)}$ can be coupled such that a.a.s.\ 
$$
\xi(P_n,k) \le \xi(C_{2(n-1)},k).
$$
\end{lemma} 
\begin{proof}
To simplify the notation, we are going to label the vertices of $P_n$ and $C_{2(n-1)}$ slightly differently than in the rest of the paper. Vertices of $P_n$ are labelled as follows: $V(P_n) = \{0, 1, \ldots, n-1\}$. On the other hand, vertices of $C_{2(n-1)}$ have labels from the set $V(C_{2(n-1)}) = \{ -(n-2), -(n-3), \ldots, -1, 0, 1, \ldots, (n-2), (n-1)\}$. The coupling will identify vertices $i$ and $-i$ on $C_{2(n-1)}$ with a vertex $i$ on $P_n$ ($i \in [n-2]$); vertex $0$ and $n-1$ on the cycle will be mapped to $0$ and, respectively, $n-1$ on the path---see Figure~\ref{fig:coupling}. There is a slight complication with making sure the agents start the process from a uniform distribution on the corresponding set of nodes. Because of that the result holds only a.a.s.\ and we need an assumption that $k=o(n)$.

\begin{figure}
\begin{center}
\begin{tikzpicture}[scale=0.7,node distance=9cm]
\coordinate (1) at (6,0);
\coordinate (2) at (4,1.5);
\coordinate (3) at (2,1.5);
\coordinate (4) at (-2,1.5);
\coordinate (5) at (-4,1.5);
\coordinate (6) at (-6,0);
\coordinate (7) at (-4,-1.5);
\coordinate (8) at (-2,-1.5);
\coordinate (9) at (2,-1.5);
\coordinate (10) at (4,-1.5);
\coordinate (label_box_1) at (-10,0);

\filldraw[color=red!60, fill=red!5, very thick](5) circle (0.7);
\filldraw[color=red!60, fill=red!5, very thick](10) circle (0.7);

\filldraw (1) circle (3pt) node [right,scale=0.7] {$n-1$};
\filldraw (2) circle (3pt) node [above,scale=0.7] {$n-2$};
\filldraw (3) circle (3pt) node [above,scale=0.7] {$n-3$};
\filldraw (4) circle (3pt) node [above,scale=0.7] {$2$};
\filldraw (5) circle (3pt) node [above,scale=0.7] {$1$};
\filldraw (6) circle (3pt) node [left,scale=0.7] {$0$};
\filldraw (7) circle (3pt) node [below,scale=0.7] {$-1$};
\filldraw (8) circle (3pt) node [below,scale=0.7] {$-2$};
\filldraw (9) circle (3pt) node [below,scale=0.7] {$-(n-3)$};
\filldraw (10) circle (3pt) node [below,scale=0.7] {$-(n-2)$};
\filldraw (label_box_1) circle (0pt) node [right] {$C_{2(n-1)}$};

\draw[thick,black] (1) -- (2);
\draw[thick,black] (2) -- (3);
\draw[dotted,black] (3) -- (4);
\draw[thick,black] (4) -- (5);
\draw[thick,black] (5) -- (6);
\draw[thick,black] (6) -- (7);
\draw[thick,black] (7) -- (8);
\draw[dotted,black] (8) -- (9);
\draw[thick,black] (9) -- (10);
\draw[thick,black] (10) -- (1);

\coordinate (0a) at (-6,-4);
\coordinate (1a) at (-4,-4);
\coordinate (2a) at (-2,-4);
\coordinate (3a) at (2,-4);
\coordinate (4a) at (4,-4);
\coordinate (5a) at (6,-4);
\coordinate (label_box_2) at (-10,-4);

\filldraw[color=red!60, fill=red!5, very thick](1a) circle (0.7);
\filldraw[color=red!60, fill=red!5, very thick](4a) circle (0.7);

\filldraw (0a) circle (3pt) node [below,scale=0.7] {$0$};
\filldraw (1a) circle (3pt) node [below,scale=0.7] {$1$};
\filldraw (2a) circle (3pt) node [below,scale=0.7] {$2$};
\filldraw (3a) circle (3pt) node [below,scale=0.7] {$n-3$};
\filldraw (4a) circle (3pt) node [below,scale=0.7] {$n-2$};
\filldraw (5a) circle (3pt) node [below,scale=0.7] {$n-1$};
\filldraw (label_box_2) circle (0pt) node [right] {$P_n$};

\draw[thick,black] (0a) -- (1a);
\draw[thick,black] (1a) -- (2a);
\draw[dotted,black] (2a) -- (3a);
\draw[thick,black] (3a) -- (4a);
\draw[thick,black] (4a) -- (5a);

\draw[dashed,gray] (2) -- (10);
\draw[dashed,gray] (3) -- (9);
\draw[dashed,gray] (4) -- (8);
\draw[dashed,gray] (5) -- (7);

\draw[dashed,gray] (0a) -- (6);
\draw[dashed,gray] (1a) -- (7);
\draw[dashed,gray] (2a) -- (8);
\draw[dashed,gray] (3a) -- (9);
\draw[dashed,gray] (4a) -- (10);
\draw[dashed,gray] (5a) -- (1);

\end{tikzpicture}
\end{center}
\caption{Coupling between $C_{2(n-1)}$ and $P_n$. Agents walking on the cycle and their avatars walking on the path are in red cycles.}
\label{fig:coupling}
\end{figure}
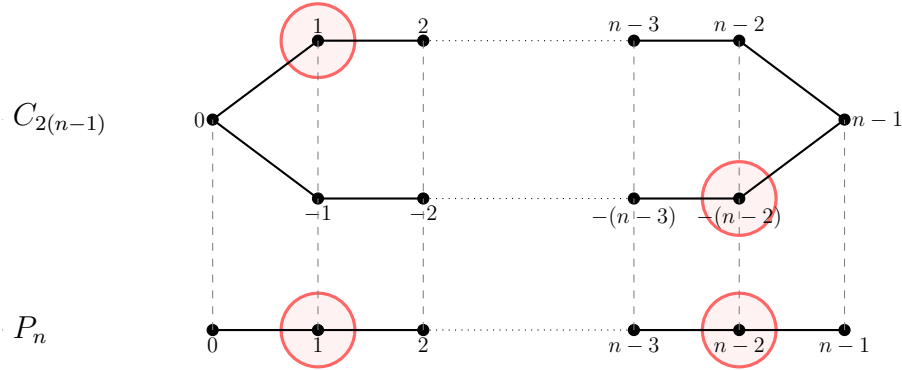

Agents start independently and uniformly at random on $C_{2(n-1)}$, as they should. Each agent, independently, becomes \emph{unusual} with probability $1/n$. Unusual agents put their \emph{avatars} on $P_n$ at vertex $0$ with probability $1/2$, and on vertex $n-1$ otherwise. Usual (that is, not unusual) agents that start at vertex $i$ on $C_{2(n-1)}$, place their avatars on $P_n$ at vertex $|i|$. It is easy to see that avatars are distributed uniformly at random on $P_n$; in particular, the probability that a given agent puts her avatar at vertex $0$ on the path is equal to
$$
\frac {1}{n} \cdot \frac {1}{2} + \left( 1 - \frac {1}{n} \right) \cdot \frac {1}{2(n-1)} = \frac {1}{2n} + \frac {1}{2n} = \frac {1}{n}.
$$

Since it is assumed that $k=o(n)$, the expected number of unusual agents is equal to $k/n = o(1)$ and so a.a.s.\ there is no unusual agent by the first moment method. If there is at least one unusual agent, then we simply stop the coupling and claim no bound for the two random variables. On the other hand, if there is no unusual agent, then agent occupying vertex $i$ on $C_{2(n-1)}$ has her avatar placed on vertex $|i|$ on $P_n$ and we may continue. Agents start walking randomly on the cycle and their avatars will follow them accordingly, that is, if an agent moves from vertex $a$ to vertex $b$ on the cycle, her avatar moves from vertex $|a|$ to vertex $|b|$ on the path. It is clear that avatars also perform independent random walks but on the path instead of the cycle. The two corresponding random walks are coupled but the broadcasting processes are performed independently on both graphs. By coupling, if two agents meet, then their avatars also meet but the converse might not be true---avatars meeting on the path might correspond to agents that occupy different vertices (again, see Figure~\ref{fig:coupling}). Hence, when all agents become green, then all avatars must be green too. This establishes the desired relationship between $\xi(P_n,k)$ and $\xi(C_{2(n-1)},k)$, and the proof of the lemma is finished.
\end{proof}

\subsection{Large $k$}\label{pathslargek}

In this section, we show that a.a.s.\ $\xi(P_n,k) = \Theta(n)$, provided that $k$ is sufficiently large, that is, $k \ge 50 n \ln n$. The coupling (Lemma~\ref{lem:coupling}) cannot be applied to this range of parameter $k$ but the proof works for cycles as well as paths. This proves part~(e) of Theorem~\ref{thm:main}. 

\medskip

Let us start with the following useful observation. Suppose that an agent starts at vertex $i \in [n]$. For a given $t \in \Nn \cup \{0\}$ and $j \in [n]$, let $P_t(i,j)$ be the probability that the agent occupies vertex $j$ at round $t$. Then the following holds.

\begin{lemma}\label{lem:symmetry}
For any $t \in \Nn \cup \{0\}$ and $i,j \in [n]$, we have that
\begin{equation}\label{eq:symmetry}
P_t(i,j) = P_t(j,i),
\end{equation}
provided that $i \notin \{1,n\}$ and $j \notin \{1,n\}$. More importantly, we always have that
\begin{equation}\label{eq:symmetry2}
\frac {P_t(j,i)}{2} \le P_t(i,j) \le 2 P_t(j,i).
\end{equation}
\end{lemma}

\begin{proof}
The lemma is an instant corollary of the fact that the associated simple random walk is reversible, see~\cite[Section~1.6]{book_Yuval}, and that the stationary distribution is near uniform. Indeed, since walking on a graph is reversible,
$$
\pi(i) P_t(i,j) = \pi(j) P_t(j,i),
$$
where $\pi(x)$ is the stationary distribution. The conclusion follows from the fact that
$$
\pi(x) = \frac {\deg(x)}{2|E|} = 
\begin{cases}
\frac {1}{2(n-1)} & \text{ if } x \in \{1,n\},\\
\frac {2}{2(n-1)} & \text{ otherwise.} 
\end{cases}
$$
The proof of the lemma is finished.
\end{proof}

\noindent \textbf{Adjustment to cycles}: The lemma holds for cycles. In fact, property~(\ref{eq:symmetry}) holds for \emph{all} $i$ and $j$ (since $\pi(x)$ is uniform on cycles) and so the weaker property~(\ref{eq:symmetry2}) trivially holds.

\medskip

We will now show that there are plenty of agents on each vertex at any round of the process, provided that it ends in at most $n$ rounds. 

\begin{lemma}\label{lem:many_agents}
Consider the process on a path $P_n$ with $k = k(n) \ge 50 n \ln n $ agents. Then, a.a.s.\ the following holds: for any $t \in [n] \cup \{0\}$ and any $j \in [n]$, the number of agents occupying vertex $j$ at round $t$ is at least $12 \ln n$. 
\end{lemma}
\begin{proof}
Fix any $t \in [n] \cup \{0\}$ and any $j \in [n]$, and let us concentrate on a given agent $A$. Let $B(i)$ be the event that agent $A$ starts at vertex $i$, and let $C(j)$ be the event that agent $A$ occupies vertex $j$ at time $t$. It follows that 
$$
\Prob \big( C(j) \big) = \sum_{i \in [n]} \Prob \big( C(j) \wedge B(i) \big) = \sum_{i \in [n]} \Prob \big( C(j) ~|~ B(i) \big) \cdot \Prob \big( B(i) \big).
$$
Since agent $A$ starts on a vertex selected uniformly at random from $V$, $\Prob \big( B(i) \big) = 1/n$. After noticing that $\Prob \big( C(j) ~|~ B(i) \big)$ is exactly $P_t(i,j)$, we get from Lemma~\ref{lem:symmetry} that
$$
\Prob \big( C(j) \big) = \frac {1}{n} \sum_{i \in [n]} P_t(i,j) \ge \frac {1}{2n} \sum_{i \in [n]} P_t(j,i) = \frac {1}{2n},
$$
as, trivially, $\sum_{i \in [n]} P_t(j,i) = 1$.

Since $k \ge 50 n \ln n$ agents select their starting points independently and perform independent random walks afterwards, the number of agents occupying vertex $j$ at round $t$ can be stochastically lower bounded by the random variable $X \sim \Bin( 50 n \ln n, 1/(2n))$. Note that $\E[X] = 25 \ln n$ and so it follows from the Chernoff inequality~(\ref{eq:chern}) applied with $\eps = 1/2$ that
\begin{eqnarray*}
\Prob \big( X \le 12 \ln n \big) &\le& \Prob \left( |X-\E[X]| \ge \frac 12 \ \E[X] \right) \\
&\le& 2 \exp \left( - \frac {1}{12} \ \E[X] \right) = 2 n^{-25/12} = o(n^{-2}).
\end{eqnarray*}
Since there are $n+1$ choices for $t$ and $n$ choices for $j$, the desired property fails for some pair of $t$ and $n$ with probability at most $n(n+1) \cdot o(n^{-2}) = o(1)$ and so the desired property holds a.a.s.\ and the proof is finished.
\end{proof}

\noindent \textbf{Adjustment to cycles}: Since Lemma~\ref{lem:symmetry} holds also for cycles, the exact same proof of the above lemma extends to cycles.

\medskip

Now, we are ready to show both an upper and a lower bound for $\xi(P_n, k)$ for $k = k(n) \ge 50 n \ln n$. 

\begin{theorem}\label{thm:large_k}
For any $k = k(n) \ge 50 n \ln n$, a.a.s. 
$
\lfloor n/2 \rfloor \le \xi(P_n, k) \le n-1.
$ 
\end{theorem}
\begin{proof}
By Lemma~\ref{lem:many_agents}, since we aim for a result that holds a.a.s., we may assume that for any round $t \in [n] \cup \{0\}$ and any vertex $j \in [n]$, the number of agents occupying vertex $j$ at round $t$ is at least $12 \ln n$. We will say that a vertex $j$ is green if it is occupied by green agents; otherwise, it is white, that is, it is occupied by white agents. In particular, this means that in the very first round ($t=0$) there is precisely one green vertex.

Suppose that at the end of some round $t \in [n-2] \cup \{0\}$, some vertex $j$ is green whereas a neighbouring vertex $i$ ($i \in \{ j-1, j+1\}$) is white. Since there are at least $12 \ln n$ green agents occupying $j$ at the end of round $t$, the probability that $i$ stays white in round $t+1$ is at most 
$$
(1/2)^{12 \ln n} = \exp \big( - 12 (\ln 2) \ln n \big) \le n^{-8} = o(n^{-1}).
$$
After applying this argument $n-1$ times, we get that a.a.s.\ at round $t \in [n-1] \cup \{0\}$ vertices at distance at most $t$ from the initial green vertex are green. In particular, a.a.s.\ all vertices become green in at most $n-1$ rounds, and so the desired upper bound holds. In fact, conditioning on the event that the initial green vertex is vertex $j \in [n]$, we get that a.a.s.\ $\xi(P_n, k) = \max \{ j-1, n-j \}$. Since $\max \{ j-1, n-j \} \ge \lfloor n/2 \rfloor$, the desired lower bound holds too, and the proof is finished.
\end{proof}


\medskip

\noindent \textbf{Adjustment to cycles}: The same argument works for cycles. Due to the symmetry, a.a.s.\ $\xi(C_n, k) = \lfloor n/2 \rfloor$.

\subsection{Walking on Integers}

Let us take a short break from our problem and briefly discuss a closely related and classical problem: walking on integers. The \emph{simple random walk on $\Zz$} starts with $X_0 = 0$ and in each round $t \in \Nn$, $X_t = X_{t-1}-1$ with probability $1/2$; otherwise, $X_t = X_{t-1}+1$. Alternatively, the \emph{lazy simple random walk on $\Zz$} starts with $X_0 = 0$ and in each round $t \in \Nn$, $X_t = X_{t-1}-1$ with probability $1/4$, $X_t = X_{t-1}+1$ with probability $1/4$, and  $X_t = X_{t-1}$ otherwise. 

It is easy to see that the sequence $X_0, X_1, \ldots$ is a martingale. In particular, the Hoeffding-Azuma inequality~\eqref{eq:h-a} can be applied to show that for small values of $t$, $X_t$ has to be relatively close to zero a.a.s. On the other hand, if $t$ is large, then a.a.s.\ $X_t$ moves away from the origin. We will need this well-known observation to establish some of our bounds. We provide the proof for completeness but for more details we direct the reader to, for example,~\cite{book_Yuval} or any other book on random walks. 

Let us first concentrate on the simple random walk. Observe that $X_t$ and $t$ are of the same parity, that is, $t-X_t$ is even. Provided that $t-a$ is even, there are $\binom{t}{\frac {t-a}{2}}$ walks of length $t$ from $0$ to $a$. Combining the two observations together we get that for any $-t \le a \le t$
$$
\Prob \left( X_t = a \right) = 
\begin{cases}
\binom{t}{\frac {t-a}{2}} 2^{-t} & \text{if $t-a$ is even}, \\
0 & \text{otherwise.}
\end{cases}
$$
It follows that for any $-t \le a \le t$, we have
$$
\Prob \left( X_t = a \right) \le \binom{t}{\lceil t/2 \rceil} 2^{-t} = \frac {t!} {\lceil t/2 \rceil ! \lfloor t/2 \rfloor!} 2^{-t} \sim \sqrt{ \frac {2}{\pi t} },
$$
where the asymptotic bound follows from \emph{Stirling's formula} ($t! \sim \sqrt{2 \pi t} (t/e)^t$). Similarly, for the lazy simple random walk, we get that for any $-t \le a \le t$
\begin{eqnarray*}
\Prob \left( X_t = a \right) &\le& \sum_{s = 0}^t \ \Prob \big( \Bin(t,1/2) = s \big) \cdot \binom{s}{\lceil s/2 \rceil} 2^{-s}.
\end{eqnarray*}
(Variable $s$ in the above formula controls the number of rounds the walk actually moves.) Chernoff's bound~\eqref{eq:chern} applied with $\eps = 1/t^{1/3}$ implies that 
\begin{eqnarray*}
\Prob \left( X_t = a \right) &\le& o(1/t) +  \sum_{s = t/2-t^{2/3}}^{t/2+t^{2/3}} \ \Prob \big( \Bin(t,1/2) = s \big) \cdot \binom{s}{\lceil s/2 \rceil} 2^{-s} \\
&\le& o(1/t) + (1+o(1)) \sum_{s = t/2-t^{2/3}}^{t/2+t^{2/3}} \ \Prob \big( \Bin(t,1/2) = s \big) \cdot \sqrt{ \frac {2}{\pi s} } \\
&\le& o(1/t) + (1+o(1)) \sqrt{ \frac {2}{\pi (t/2)} } \sum_{s = t/2-t^{2/3}}^{t/2+t^{2/3}} \ \Prob \big( \Bin(t,1/2) = s \big) \sim \sqrt{ \frac {4}{\pi t} }.
\end{eqnarray*}
Hence, regardless of whether we deal with lazy random walks or not, for $t$ large enough and any $a \ge 1$ we have
\begin{equation}\label{eq:random_walk}
\Prob \left( |X_t| < a \right) \le \frac {4a}{\sqrt{t}}\,.
\end{equation} 

We will also need the following result on the hitting time defined as follows:
$$
\tau_a = \min \{ t \ge 0 : X_t = a \},
$$      
that is, $\tau_a$ is the first time the walk hits $a$. Using the \emph{reflection principle}, one can show that 
\begin{equation}\label{eq:hitting_time}
\Prob (\tau_a > t) = \Prob( -a < X_t \le a) \le \frac {4a}{\sqrt{t}}\,.
\end{equation}
(See Lemma~2.21 in~\cite{book_Yuval} that applies to both lazy and non-lazy simple random walks.)

\subsection{Small $k$}\label{sec:small_values_of_k}

Let $\omega=\omega(n)$ be any function that tends to infinity as $n \to \infty$. In this section, we show that a.a.s.\ $\xi(P_n,k) \ge n^2 / (\omega k^2 \ln k)$ and $\xi(C_n,k) \le n^2 \omega$, provided that $k \le \omega \ln n$. These two bounds, together with the coupling (Lemma~\ref{lem:coupling}), prove parts~(a) and~(b) of Theorem~\ref{thm:main}. 

\medskip

We start by proving an upper bound for $\xi(C_n,k)$. It is a strong bound for small values of $k$ but a weak one for large values of $k$. Recall that, in particular, $\xi(C_n,k) = O(n)$ for $k \ge 50 n \ln n$. However, since it holds for all values of $k$, we state it here in full generality.

\begin{theorem}\label{thm:small_k_upperbound}
Let $\omega=\omega(n)$ be any function that tends to infinity as $n \to \infty$. 
For any $k = k(n) \ge 2$, a.a.s.\
$
\xi(C_n, k) \le n^2 \omega.
$
\end{theorem}

\begin{proof}
We will couple our process with the random walk on integers we discussed above in the most natural way. If an agent starts at vertex $i \in [n]$ on the cycle, her avatar starts at integer $X_0 = i$ on $\Zz$. If $X_t$ increases, then the agent occupying vertex $i < n$ moves to $i+1$ and she moves to $0$ if she occupies vertex $n$. Similarly, if $X_t$ decreases, then the agent occupying vertex $i > 1$  moves to $i-1$ and she moves to $n$ if she occupies vertex $1$.

Concentrate on the initial green agent and an arbitrary white agent that are at distance $d$ from each other. Our goal is to control random variable $Y_t$, the ``distance'' between the corresponding avatars walking on $\Zz$. We initiate the auxiliary process with $Y_0 = d$ and for each $t \in \Nn$, $Y_{t} = Y_{t-1} + 2$ with probability $1/4$, $Y_{t} = Y_{t-1} - 2$ with probability $1/4$, and $Y_{t} = Y_{t-1}$ otherwise. (Hence, effectively, it is a lazy random walk.) The distance between the two avatars is $|Y_t|$.

Consider the first $t = n^2 \omega$ rounds. Our bound~(\ref{eq:random_walk}) applied with $a=2n+1$ implies that a.a.s.\ $|Y_t| \ge 2n+1$. 
But this implies that the two agents met at some point (when $|Y_t| = n$ or $|Y_t| = n+1$) and then met again (when $|Y_t| = 2n$ or $|Y_t| = 2n+1$), after making in the meantime everyone else green.
\end{proof}

\noindent \textbf{Adjustment to paths}: This is the only situation when the argument for cycles \emph{cannot} be easily adjusted to deal with paths. We provide an independent, direct argument if one does not want to use the coupling between the two families of graphs.

\begin{theorem}
Let $\omega=\omega(n)$ be any function that tends to infinity as $n \to \infty$. 
For any $k = k(n) \ge 2$, a.a.s.\
$
\xi(P_n, k) \le n^2 \omega.
$ 
\end{theorem}

\begin{proof}
We will estimate the number of rounds needed for the initial green agent to travel to one of the endpoints of the path and then to walk to the other endpoint. We will show that a.a.s.\ it happens in at most $n^2 \omega$ rounds. This will finish the proof as it guarantees that all other agents have to meet her at some point and so all of them eventually become green. 

We will couple our process with the random walk on integers we discussed above in the most natural way. If $X_t$ increases and the agent occupies vertex $i < n$, she moves to $i+1$. Similarly, if $X_t$ decreases and the agent occupies vertex $i > 1$, she moves to $i-1$. However, if she occupies one of the endpoints of the path (vertex $1$ or vertex $n$), her move is deterministic as she is forced to stay on the path, regardless of what the random walk does. 

Consider the first $t = n^2 \omega / 2$ rounds. Our bound~(\ref{eq:random_walk}) applied with $a=n$ implies that a.a.s.\ $|X_t| \ge n$. This implies that the agent must bump into one of the endpoints (say, vertex $1$) during this time period (say, at time $T \le t$). Let us now concentrate on the next $t = n^2 \omega / 2$ rounds following time $T$ and let us restart the coupled random walk by fixing $X_T=1$. Using~(\ref{eq:hitting_time}) applied with $a=n$ we conclude that a.a.s.\ the random walk hits integer $n$ during that period of time, and so the agent has to visit vertex $n$ as well. This concludes the proof.
\end{proof}

\medskip

Let us now turn our attention to a lower bound for $\xi(P_n,k)$. 

\begin{theorem}\label{thm:upper_bound_small_k}
Let $\omega=\omega(n)$ be any function that tends to infinity as $n \to \infty$. 
For any $k = k(n) \le \omega \ln n$, a.a.s.\
$
\xi(P_n, k) \ge n^2 / (\omega \, k^ 2 \ln k ).
$ 
\end{theorem}

\begin{proof}
Note that with probability $1-O(1/\omega^{1/3}) \sim 1$, the initial green agent starts the process at distance at least $n/\omega^{1/3} = o(n)$ from both endpoints. Similarly, with probability $1-O(1/(k \omega^{1/3}))$, a given white agent starts the process at distance at least $n/(k \omega^{1/3}) = o(n/k)$ from the green agent. Hence, all white agents are at distance at least $n/(k \omega^{1/3})$ from the green agent with probability 
$$
(1-O(1/(k\omega^{1/3})))^{k-1} = 1-O(k/(k\omega^{1/3})) \sim 1. 
$$
Since we aim for a conclusion that holds a.a.s., we may assume that this property is satisfied at the end of round 0. 

Trivially, at the end of the whole process (that is, when all agents become green), at least one agent (either the one that was initially green or one of the white ones) has to move at least $n/(2 k \omega^{1/3})$ away from her initial position; otherwise, no white agent turns green. We will show that this is highly unlikely after only $n^2 / (\omega \, k^2 \ln k)$ rounds. Applying the Hoeffding-Azuma inequality~(\ref{eq:h-a}) with $b=n/(2k\omega^{1/3})$, $c_a=1$, and $t=n^2 / (\omega \, k^2 \ln k)$ implies that a given agent moves that far with probability at most
$$
2 \exp \left( - \frac {b^2}{2t} \right) = 2 \exp \left( - \frac {\omega^{1/3}}{8} \ln k \right) = o(1/k).
$$
Hence, the probability that at least one agent moves far is $O(k) \cdot o(1/k) = o(1)$, and the proof is finished.
\end{proof}

\noindent \textbf{Adjustment to cycles}:  The argument is easily adjusted for cycles. In fact, it is slightly simpler as one does not need to pay attention to the two endpoints of the path.

\subsection{Medium $k$}\label{sec:medium_values_of_k}

Let $\omega=\omega(n)$ be any function that tends to infinity as $n \to \infty$. In order to prove upper bounds in parts~(c) and~(d) of Theorem~\ref{thm:main}, we need to concentrate on $k = k(n) \gg \ln n$ and $k = k(n) < 50 n \ln n$. 

\medskip

Let us start with the following definition. We partition the set of vertices of the path $P_n$ into $b = b(n) := \lfloor k / (500 \ln n) \rfloor$ \emph{blocks}. (Note that $b \gg 1$ and $b < n/10$.) Each block consists of either $\lfloor n/b \rfloor \ge \lfloor 500 n \ln n / k \rfloor \ge 10$ or $\lceil n/b \rceil$ vertices. 

We will first adjust the proof of Lemma~\ref{lem:many_agents} to show that there are plenty of agents on each block at any round of the process, provided that it ends in at most $n^2$ rounds. The adjustment is easy and straightforward but we provide the proof for completeness. Moreover, since Lemma~\ref{lem:symmetry} holds for both paths and cycles, the lemma below holds for both families of graphs too.

\begin{lemma}\label{lem:many_agents2}
Consider the process on a path $P_n$ with $k$ agents such that $\ln n \ll k = k(n) < 50 n \ln n $. Then, a.a.s.\ the following holds: for any $t \in [n^2] \cup \{0\}$ and any $j \in [b]$, the number of agents occupying block $j$ at round $t$ is at least $100 \ln n$ and at most $1800 \ln n$. 
\end{lemma}

\begin{proof}
The number of agents occupying block $j$ at round $t$ can be stochastically lower bounded by random variable $X \sim \Bin( k, (450 n \ln n / k) /(2n))$; recall that each block has length at least $\lfloor 500 n \ln n / k \rfloor \ge (9/10) (500 n \ln n / k) = 450 n \ln n /k$ and, by Lemma~\ref{lem:symmetry}, each vertex is occupied by a given agent with probability at least $1/(2n)$. Note that $\E[X] = 225 \ln n$ and so it follows from the Chernoff inequality~(\ref{eq:chern}) applied with $\eps = 1/2$ that
\begin{eqnarray*}
\Prob \big( X \le 100 \ln n \big) &\le& \Prob \left( |X-\E[X]| \ge \frac 12 \ \E[X] \right) \\
&\le& 2 \exp \left( - \frac {1}{12} \ \E[X] \right) \le 2 n^{-225/12+o(1)} = o(n^{-3}).
\end{eqnarray*}
Since there are $n^2+1$ choices for $t$ and $b = \lfloor k / (500 \ln n) \rfloor < n/10$ choices for $j$, the desired property fails for some pair of $t$ and $j$ with probability at most 
$$
(n^2+1)(n/10) \cdot o(n^{-3}) = o(1). 
$$
It follows that the desired lower bound for the number of agents holds a.a.s.

Similarly, the number of agents occupying block $j$ at round $t$ can be stochastically upper bounded by random variable $Y \sim \Bin( k, (600 n \ln n / k)(2/n))$; recall that each block has length at most $\lceil n/b \rceil \le \lceil 501 n \ln n / k \rceil \le (11/10) (501 n \ln n / k) \le 600 n \ln n /k$ and, by Lemma~\ref{lem:symmetry}, each vertex is occupied by a given agent with probability at most $2/n$. We get that $\E[Y]=1200 \ln n$ and $Y \ge (3/2) \E[Y] = 1800 \ln n$ with probability $o(n^{-3})$. The desired upper bound holds a.a.s.\ too and the proof is finished.
\end{proof}

\noindent \textbf{Adjustment to cycles}:  As mentioned above, the above lemma holds for cycles and the proof is exactly the same since Lemma~\ref{lem:symmetry} holds for both paths and cycles.

\medskip

We are now ready to prove an upper bound. Note that the coupling (Lemma~\ref{lem:coupling}) cannot be applied to the whole range of parameter $k$. Since the arguments used to deal with cycles and paths are the same but some technicalities are slightly more involved for paths, we decided to present an argument for paths instead of using the coupling.

\begin{theorem}\label{thm:medium_k_upperbound}
Let $\omega=\omega(n)$ be any function that tends to infinity as $n \to \infty$. 
For any $k$ such that $\omega \ln n \le k = k(n) < 50 n \ln n $, a.a.s.\
$
\xi(P_n, k) = O( n^2 \ln n / k ).
$ 
\end{theorem}

\begin{proof}
Since the argument is quite involved, let us first provide a high level overview of the proof. First, we will show that the initial green agent quickly meets a white agent. Then, we will track the distance between them and show that at some point they are far apart from each other so that they are separated by at least one block. As a result, not only are these two agents green but, in particular, all agents present on that internal block are green. The final step is to show that the sequence of blocks consisting of only green agents keeps expanding, eventually reaching both endpoints of the path.

Without loss of generality, we may assume that $\omega \le \ln \ln n$. Suppose that the initial green agent starts the process at vertex $j \in [n]$. Clearly, a.a.s.\ the initial green agent is at distance at least $n/\omega^{1/5}$ from both endpoints of the path, that is, $j > n/\omega^{1/5}$ and $i < n - n/\omega^{1/5}$. The probability that no white agent starts the process at distance at most $d:=\omega n / k$ from the initial green agent is at most 
$$
\left( 1 - \frac {d}{n} \right)^{k-1} \le \exp \left( - \frac {\omega (k-1)}{k} \right) \le \exp ( -\omega / 2) = o(1).
$$
Since we aim for the statement that holds a.a.s., we may assume that initially the green agent is at distance at least $n/\omega^{1/5}$ from both endpoints and one of the white agents starts at distance at most $d = \omega n / k = o(n/\omega^{1/5})$ from her. 

Let us focus on the initial green agent and a white agent that is initially the closest to her (if there are multiple white agents with this property, pick one of them arbitrarily). Provided that they did not yet meet, our process can be coupled with the lazy simple random walk on $\Zz$ starting at $X_0=\lceil d/2 \rceil$ in such a way that the distance at round $t$ between the two agent is at most $2X_t$. It follows from~\eqref{eq:hitting_time}, applied with $a=\lceil d/2 \rceil$ and $t = \omega d^2 = \omega^3 n^2 / k^2 \le n^2 / k = o(n^2 \ln n / k)$, that a.a.s.\ they meet at some round $T \le t$. On the other hand, the Hoeffding-Azuma inequality~\eqref{eq:h-a}, applied with $t = \omega d^2$ and $b = \sqrt{\omega  t} = \omega d = \omega^2 n / k \le n / (3\omega^{1/5})$, implies that a.a.s.\ both agents are at distance at most $n/(3\omega^{1/5})$ from their initial positions when they meet and so both of them are still at a distance of at least $n/(2\omega^{1/5})$ from both endpoints of the path. As before, since we aim for the statement that holds a.a.s., we may assume that this property holds.

We continue the process for an additional $t = \sqrt{\omega} (n \ln n / k)^2 \le n^2 / \omega^{3/2}$ rounds measured from round $T$ when they met. Similarly as before, provided that they did not reach the end of the path, the process can be coupled with the lazy simple random walk on $\Zz$ starting at $X_0=0$ in such a way that the distance at round $t$ between the two agent is at least $2X_t$. It follows from~\eqref{eq:hitting_time}, applied with $a=\lceil n/b \rceil = \Theta(n \ln n / k)$ and $t = \sqrt{\omega} (n \ln n / k)^2$, that a.a.s.\ some agent reaches the end of the path or the two agents are at distance more than $2a$ at some point of that period of time. On the other hand, after applying the Hoeffding-Azuma inequality~\eqref{eq:h-a}, with $t = \sqrt{\omega} (n \ln n / k)^2$ and $b = \sqrt{\omega  t} = \omega^{3/4} (n \ln n / k) \le n / \omega^{1/4} = o( n / \omega^{1/5} )$, we get that a.a.s.\ neither agent reaches the end of the path during that period of time. We conclude that a.a.s.\ the two agents are at distance more than $2a$ at some point, that is, at some point they occupy two different blocks of the path that are separated by at least one block. A trivial but important property is that all agents occupying that internal block are green. 

Suppose that at some point of the process, all agents occupying block $i$, $2 \le i \le b$, are green but some agent occupying block $i-1$ is white. Note that each block has length at most $\lceil n/b \rceil \le (11/10) (n/b) \le 600 n \ln n / k$. Let $t = (9600 n \ln n / k)^2$. Let us concentrate on any agent occupying block $i$. It follows from equation~(\ref{eq:random_walk}), applied with $a = \sqrt{t} / 8 = 1200 n \ln n / k$ and the above $t$, that with probability at least $1/2$ the agent either reaches the endpoint of the path (at some point during the following $t$ rounds) or is at distance at least $a$ from the original place after $t$ rounds. Hence, by symmetry, with probability at least $1/4$, this agent either at some point reaches the endpoint of the path or occupies block $j \le i-1$ after $t$ rounds. In both scenarios all agents occupying block $i-1$ become green. By Lemma~\ref{lem:many_agents2}, there are at least $100 \ln n$ agents occupying block $i$. Hence the probability that no agent does the job is at most $(3/4)^{100 \ln n} = o(n^{-1})$. By symmetry, the same argument can be applied when all agents occupying block $i$, $1 \le i \le b-1$, are green but some agent occupying block $i+1$ is white.

Since there are $b \le n$ blocks, by the union bound, we get that a.a.s.\ after any period of $t$ rounds, the number of blocks with all agents being green increases and so the process is done after at most
\begin{eqnarray*}
o(n^2 \ln n / k) + t \cdot b &=& o(n^2 \ln n / k) + (9600 n \ln n / k)^2 \cdot (k/(500 \ln n)) \\
&=& O(n^2 \ln n / k)
\end{eqnarray*}
rounds and so the proof is finished.
\end{proof}

\noindent \textbf{Adjustment to cycles}:  As mentioned earlier, the same argument works for cycles and the proof is simpler as one does not need to pay attention to the two endpoints of the path. The same comment applies to the next theorems below.

\medskip

Let us now move to a lower bound. 

\begin{theorem}\label{thm:medium_k_lowerbound}
Let $\omega=\omega(n)$ be any function that tends to infinity as $n \to \infty$. 
For any $k$ such that $\omega \ln n \le k = k(n) \le n / \ln^2 n$, a.a.s.\
$$
\xi(P_n, k) = \Omega \left( \frac {n^2}{k (\ln k) (\ln n)} \right) = \Omega \left( \frac {n^2}{k \ln^2 n} \right).
$$
Moreover, for any $k$ such that $n / \ln^2 n < k = k(n) < 50n \ln n$, a.a.s.\ 
$$
\xi(P_n, k) > (1/2+o(1)) \, n.
$$
\end{theorem}
 
\begin{proof}
Assume first that $\omega \ln n \le k = k(n) \le n / \ln^2 n$. Recall that there are $b=b(n) = \lfloor k/(500 \ln n) \rfloor$ blocks (because of our assumption that $b \gg 1$ and $b \ll n / \ln^3 n$), each of length $(500+o(1)) n \ln n / k$. We assign to each block a label from $[b]$; the first block contains vertex $1$ and the last block contains vertex $n$. By symmetry, we may assume that the initial green agent starts at block $i_0 \ge b/2$. We will say that a green agent is \emph{leading} if she occupies vertex $j$ and no other green agent occupies vertex $\ell < j$. Note that leading agents may (and often do) change during the process, and there could be more than one leading agent at a given round. We will concentrate on leading agents and investigate times $t_i$ when a leading agent leaves block $i+1$ and enters block $i$ for the first time. We will show that the following property holds a.a.s.: for all values of $i \in [i_0-2]$, 
\begin{equation}\label{prop_T}
t_{i}-t_{i+1} \ge T = T(n) := \frac {n^2} {384 \, k^2 \ln k}.
\end{equation}
Since $k = k(n) \le n / (\sqrt{\ln n} \, \omega)$, we get that $T \gg 1$ (this is the reason we had to introduce an upper bound for $k$; in fact, we assumed that $k \le n / \ln^2 n$ as for larger values of $k$ we will be able to prove a stronger bound anyway). This will yield a lower bound for $\xi(P_n, k)$ as it proves that it takes at least $(i_0-1) T \ge (b/2-1)T = \Omega( n^2 / (k \ln k \ln n))$ steps for the leading agent to reach the first block. Reaching the first block is needed as, by Lemma~\ref{lem:many_agents2}, a.a.s.\ there are agents in that block that by the definition of the leader are still white. 

Let us fix any $i \in [i_0-2]$ and investigate the situation at time $t_{i+1}$ when a leading agent enters block $i+1$. It follows from Lemma~\ref{lem:many_agents2} that there are at most $1800 \ln n$ agents occupying block $i+1$. Since each block has length $(500+o(1)) n \ln n / k$, there must be a gap between two agents occupying that block that is of length at least $g:=n/(4k)$. Trivially, between time $t_{i+1}$ and $t_i$ at least one agent has to move at least $g/2 = n/(8k)$ from her position at time $t_{i+1}$; otherwise, no agent crosses the middle vertex of the gap and so no green agent enters block $i$. The probability that at least one agent crosses the gap during $T$ rounds is, by Hoeffding-Azuma inequality~(\ref{eq:h-a}) applied with $b=g/2$ and $t=T$, at most
$$
k \cdot 2 \exp \left( - \frac {(g/2)^2}{2T} \right) = 2k \exp \left( - \frac { n^2/(64 k^2) }{ 2 n^2 / (384 k^2 \ln k) } \right) = 2 / k^2.
$$
Property~(\ref{prop_T}) holds for a given $i$ with probability $1-o(1/k)$. Since $i_0 \le b \le k$, by the union bound, it holds a.a.s.\ for all $i \in [i_0-2]$. The desired lower bound holds for this range of $k$. 

The argument for $n / \ln^2 n < k = k(n) < 50n \ln n$ is straightforward. As before, by symmetry we may assume that the initial green agent starts at block $i_0 \ge b/2$, that is, she is at distance at least $d = d(n) := n/2 - (500+o(1)) n \ln n / k \sim n/2$ from the first block. By Lemma~\ref{lem:many_agents2}, we may assume that the first block is always occupied by some agents. Trivially (and deterministically), it takes at least $d(n)$ steps for a leading agent to reach that block which yields the desired lower bound. The proof is finished. 
\end{proof}

In fact, in order to get a slightly stronger lower bound, one may use the fact that agents occupying blocks that are further away from the gap have to move more than $g/2$. Agents that are far can be dealt with easily. Hence, the union bound can be taken over $\Theta(\ln n)$ agents occupying close blocks instead of $k$ agents. This would improve the bound by a multiplicative factor of $\Theta( \ln k / \ln \ln n )$. However, since such a lower bound does not match the upper bound we proved above, we stayed with an easier proof of a slightly weaker bound. 

\section{Closing the Gap and the Meet-Exchange Process}

Many of the bounds in this paper are proved in a ``local'' fashion by proving concentration bounds for the ``moving parts'', whether these parts are trajectories of walks or numbers of agents in blocks. Once these bounds are established, one essentially needs to consider the worst case bound that holds with desired probability and treats the process as deterministic. This is a classic and natural approach. However, for this problem it will never be enough to establish tight bounds.

In order to obtain tight bounds, one needs to use a ``global'' approach. Such approach was successfully applied in~\cite{ref6} and then in~\cite{ref4} to analyze similar models. It is possible that it could be used again to get tight bounds for the broadcasting time for cycles and paths (possibly grids too). We leave it as an open problem.

\medskip

Finally, let us mention the \emph{Meet-Exchange} process that was recently introduced and studied in~\cite{ref3}. This process is closely related to our process but there are a few differences. In Meet-Exchange, agents are placed independently on the vertices of a graph according to the stationary distribution $\pi$ instead of selecting starting points uniformly at random. A message is left on one of the vertices and needs to be picked up the agents before they start passing it to each other. Moreover, agents perform a lazy random walk to avoid a problem of agents never meeting if the graph is bipartite (in the broadcasting time we avoid this issue by passing a message if agents go through the same edge but in the opposite directions---see below for more details). Despite these differences, it is quite possible that the bounds proved for the broadcasting time can also be proved for the Meet-Exchange process. Indeed, the stationary distribution is uniform on the cycle and almost uniform on the path. The message is picked by the agents quickly. In~\cite{ref11} it was shown that it takes $O((n/k)^2 \ln^2 k)$ rounds in expectation and so it is negligible unless $k$ is very small. Addressing the fact that agents in the Meet-Exchange perform a lazy random walk should also be possible but disregarding passing a message by crossing agents seems to be the most challenging task. Having said that, since the considered graphs are strongly recurrent, once agents meet they typically do so a few times.With more work one should be able to overcome these technicalities. We also leave it as an open problem.

\bigskip

We would like to thank anonymous reviewers for pointing these papers and the ``global'' approach to us and for many other valuable comments that substantially improved the quality of this paper.

\end{document}